\documentclass[11pt]{amsart}

\usepackage{amsmath,amscd,amssymb,latexsym,color}

\newcommand{\mat}{\begin{pmatrix}}
\newcommand{\emat}{\end{pmatrix}}

\renewcommand{\t}{\tau}

\renewcommand{\a}{\alpha}
\renewcommand{\b}{\beta}
\renewcommand{\i}{\infty}
\newcommand{\G}{\Gamma}
\newcommand{\g}{\gamma}

\newcommand{\C}{\mathbb C}

\renewcommand{\H}{\mathbb H}

\newtheorem{thm}{Theorem}
\newtheorem{lem}[thm]{Lemma}
\newtheorem{cor}[thm]{Corollary}
\newtheorem{prop}[thm]{Proposition}

\theoremstyle{definition}

\numberwithin{equation}{section}
\numberwithin{thm}{section}

\begin{document}

\title[Divisibility Properties of the Fourier Coefficients of (Mock) Modular Functions]{Divisibility Properties of the Fourier Coefficients of (Mock) Modular Functions and Ramanujan}

\author{Soon-Yi Kang}
\address{Department of Mathematics, Kangwon National University, Chuncheon, 24341, Korea} \email{sy2kang@kangwon.ac.kr}

\dedicatory{Dedicated to my teacher Bruce C. Berndt on his 80th Birthday.}

\begin{abstract} We survey divisibility properties of the Fourier coefficients of modular functions inspired by Ramanujan. Then using recent results of the generalized Hecke operator on harmonic Maass functions and known divisibility of Fourier coefficients of modular functions, we establish congruence relations of the Fourier coefficients of certain modular functions and mock modular functions of various levels.
\end{abstract}
\maketitle

\renewcommand{\thefootnote}%
             {}
 {\footnotetext{
 2010 {\it Mathematics Subject Classification}: 11F03, 11F12, 11F25, 11F30, 11F33
 \par
 {\it Keywords}: congruence, Fourier coefficients, Hecke operator, modular function, mock modular form, Ramanujan
}}

\section{Introduction}

Fourier coefficients of modular forms have rich arithmetic properties and give a wide range of applications. It is in the work of S.~Ramanujan where we find important such examples and the rudiment of the arithmetic theory of modular forms. 

In his paper \cite{R19, R20, R21}, Ramanujan stated and proved his three famous congruences for the partition function $p(n)$, namely,
\begin{eqnarray}
p(5n + 4) \equiv 0 & \pmod 5,\\
p(7n + 5) \equiv 0 &\pmod 7, \\
p(11n + 6)\equiv 0 &\pmod{11}, \label{R11}
\end{eqnarray}
where $n$ is any non-negative integer.  The partition function is closely related with the $\t$-function that was introduced by Ramanujan in his important paper \cite{R16} 
via
$$\sum_{n=1}^\i \t(n)q^n=q\prod_{n=1}^\i (1-q^n)^{24}=\left(\prod_{n=1}^\i(1-q^n)^{25}\right)\sum_{n=0}^\i p(n)q^{n+1}.$$
He proved a congruence for $\t(n)$ as well:
\begin{equation}
\t(n)\equiv \sigma_{11}(n)\pmod{691},
\end{equation}
where $\sigma_k(n)$ is the sum of $k$th powers of divisors of $n$. Ramanujan's observations on $\t(n)$ such as
$$\t(p^{\a+2})=\t(p)\t(p^{\a+1})-p^{11}\t(p^\a)$$
 for prime $p$ and $\a\geq 0$ led to the development of Hecke theory.
The proofs of (1.1)--(1.3) in \cite{R21} employ Eisenstein series.  In \cite{Berndt2007}, B.~C.~Berndt also uses Eisenstein series and gives simple and unified proofs for (1.1)--(1.3).
%with the proof of (1.3) being precisely that of J.~M.~Rushforth \cite{Rush}.
In fact, the proofs in \cite{R21} were extracted from an unpublished manuscript
of Ramanujan on $p(n)$ and $\t (n)$ by G.~H.~Hardy after Ramanujan died in 1920. 
Berndt and K.~Ono \cite{BerO} edited this manuscript with an extensive commentary and references to the literature.

Ramanujan \cite{R19} conjectured that
there are further congruences for $p(n)$ modulo powers of the primes $5$, $7$, and $11$ and   
%More precisely,
%\begin{conj}[Ramanujan] Let $q=5,7$ or $11$. If $24m\equiv 1 \pmod{q^n}$, then $p(m)\equiv 0 \pmod{q^{n}}$.
%s
%\end{conj}
left a sketch of proofs of the conjecture for all powers of 5 and 7 in his manuscript \cite{BerO}.
G.~N.~Watson \cite{Wat} completed the proofs. The case for all powers of $11$ was proved by  A.~O.~L. Atkin \cite{Atkin1967}, who
followed J. Lehner's approach developed in \cite{Lehner1943, Lehner1950}.
For more history and discoveries on congruences  for $p(n)$ including two groundbreaking work by Atkin \cite{Atkin1968} and S. Ahlgren and Ono \cite{AhlO}, see \cite{Berndt2000, Ono04, Ono08, PR}. 

It is too extensive to discuss all divisibility properties of Fourier coefficients of modular forms inspired by Ramanujan. We rather focus on the divisibility of those relating with modular functions. 
It was Lehner who started investigating the divisibility properties of the Fourier coefficients of a modular function, in particular, the modular $j$-invariant. For $\t$ with $\textrm{Im}\t>0$ and $q=\exp(2\pi i \t)$, the classical modular $j$-invariant
$j(\t)=q^{-1}+744+\sum_{n=1}^\infty b(n)q^n$
is a modular function on the full modular group, that is, a weight $0$ weakly holomorphic modular form of level $1$.  The Fourier coefficients of this fundamental function in modular form theory and number theory are now known as the graded dimensions of the Monster module. Expanding the methods used in his proofs of Ramanujan's partition congruences, Lehner showed in \cite{Lehner1949a} that
\begin{eqnarray}
 &b(5^{\a+1} n) \equiv 0 & \pmod{5^\a},\\ 
 &b(7^\a n) \equiv 0 & \pmod{7^\a}, \\
 &b(11^2 n) \equiv 0 & \pmod{11^2},  \label{l11}
\end{eqnarray}for any positive integers $n$ and $\a$.
In \cite{Atkin1967}, Atkin generalized \eqref{l11} to all positive powers $\a$:
\begin{equation}\label{j11}b(11^\a n) \equiv 0  \pmod{11^\a}.\end{equation} 
As he claimed in \cite[p.22]{Atkin1967}, congruence in \eqref{j11} is best possible in the sense that $b(11^\a) \not\equiv 0  \pmod{11^{\a+1}}.$
Lehner \cite{Lehner1949} further showed that the coefficients $b(n)$ are highly divisible by small primes dividing $n$ by proving that
\begin{equation}\label{small}
b(2^\a 3^\b 5^\g 7^\delta  n) \equiv 0  \pmod{2^{3\a+8} 3^{2\b+3} 5^{\g+1} 7^\delta}
\end{equation} for any positive integers $\a$, $\b$, $\g$, and $\delta$.  He also claimed that the Fourier coefficients of certain modular functions of higher level also satisfy the same divisibility properties.  
Recently, these results have been generalized to a modular function of level of genus zero. M.~Griffin \cite{Griffin} showed that similar congruences with \eqref{small} hold for all elements of the canonical basis $j_m(\t):=q^{-m} + O(q)$ for the space of weakly holomorphic modular functions on the full modular group. Note that $j_0(\t)=1$ and $j_1(\t)=j(\t)-744$.  Applying \eqref{j11} into the method used in \cite{Griffin}, one can easily obtain a prime 11 analogue to the results for smaller primes in \cite{Griffin}. 

\begin{thm}\label{tb11} Let $j_{m}(\t)=q^{-m}+\sum_{n=1}^\i b(m,n)q^n$ be an element of the canonical  basis for the space of weakly holomorphic modular functions on the full modular group. If $m=11^\a m'$ for a non-negative integer $\a$ and $(11,m')=1$, then for any positive integers $n$ and $\b>\a$,
\begin{equation}\label{eb11}
b(11^{\a} m',11^{\b} n)\equiv 0 \pmod{11^{(\b-\a)}}.
\end{equation}
\end{thm}

N.~Anderson and P.~Jenkins \cite{AJ} generalized Lehner's claim on modular functions of higher level by proving the following theorem:
\begin{thm}\cite[Theorem 2]{AJ}\label{ajp} Let $p\in\{2,3,5,7\}$, and let
$$f_{p,m}(\t)=q^{-m}+\sum_{n=0}^\i {b_{p}(m,n)q^n}$$
be an element of the canonical basis for the space of level $p$ modular functions which are holomorphic at $0$. If $m=p^\a m'$ for a non-negative integer $\a$ and $(m',p)=1$, then for any positive integer $n$ and $\b>\a$,
\begin{eqnarray}\label{eajp}
%\begin{align}
\begin{split}
 &b_2(2^{\a}m',2^\b n) \equiv 0 & \pmod{2^{3(\b-\a)+8}},\\ 
 &b_3(3^{\a}m',3^\b n) \equiv 0 & \pmod{3^{2(\b-\a)+3}},\\
 &b_5(5^{\a}m',5^\b n) \equiv 0 & \pmod{5^{(\b-\a)+1}},\\ 
 &b_7(7^{\a}m',7^\b n) \equiv 0 & \pmod{7^{(\b-\a)}}.
 \end{split}
%\end{align}
\end{eqnarray}
\end{thm}
Divisibility properties of Fourier coefficients of modular functions of level of genus zero are further investigated in the series of the papers by Jenkins and his collaborators in \cite{IJW, JT15, JT18}.
They showed that many of the Fourier coefficients of the basis elements for the space of modular functions of level $N$ of genus zero are divisible by high powers of the prime dividing the level $N$.

Unlike modular curves of level $2$, $3,$ $5$ and $7$, the curve of level $11$ is of genus $1$ not of genus $0$. Hence there is no Hauptmodul on this curve. In order to prove  congruence relations relating with $11$, Atkin and Lehner, each found a pair of two generators to construct a basis for the space of modular functions of level $11$.   In \cite{JKK-hecke}, the author with D.~Jeon and C.~H.~Kim constructed reduced form bases for the spaces of the modular functions that are holomorphic away from the cusp at $i\i$ for all levels. In particular, each basis element function with level $N$ of genus 1 is given by
\begin{equation}\label{fN}
\begin{cases}
f_{N,0}(\t)=1,\\
f_{N,m}(\t)=q^{-m}+a_{N}(m,-1)q^{-1}+\sum_{n=1}^\i {a_{N}(m,n)q^n},\quad (m\geq 2)
\end{cases}
\end{equation}
with integer coefficients.  Using Atkin's basis and its properties discovered in \cite{Atkin1967}, we establish a divisibility property of the Fourier coefficients of $f_{11,m}(\t)$ modulo powers of $11$: 
\begin{thm}\label{tf11} Let $f_{11,m}(\t)$ be an element of the basis for the space of level $11$ modular functions which are holomorphic at $0$ with the Fourier expansion in \eqref{fN}.
If $m=11^\a m' \geq 2$ for a non-negative integer $\a$ and $(11,m')=1$, then for any positive integers $n$ and $\b>\a$, 
\begin{equation}\label{e11}
a_{11}(11^\a m',11^{\b} n)\equiv 0 \pmod{11^{(\b-\a)}}.
\end{equation}
\end{thm}
According to numerical experiments, such beautiful and simple congruences appearing in \eqref{small} through \eqref{e11} do not seem to hold for levels $17$ or $19$.  However, a very general congruence relation of Fourier coefficients of modular functions of level of positive genus is found in \cite[Theorem 1.11]{JKK-hecke}. We state only its genus 1 version here: 
\begin{thm}\label{Tcong}
Let $N\in\{11,14,15,17,19,20,21,24,27,32,36,49\}$. Then for any prime $p$ and positive integers $\a$, $m$ and $n$ with $m$, $n\not\equiv 0\pmod p$, we have
\begin{equation}\label{eqcong}a_{N}(p^\a m,n)+a_{N}(m,-1)a_{N}(p^\a,n)\equiv 0\pmod{p^\a}.\end{equation}
\end{thm}

Ramanujan's mock theta functions are holomorphic parts of harmonic Maass forms. In this spirit and following Zagier, we call the holomorphic part of a harmonic Maass function (i.e., a weight zero harmonic Maass form) by a mock modular function.
For each level $N$, there are certain analytic continuations of Niebur-Poincar\'e series which form a basis for the space of harmonic Maass functions of level $N$. We denote these basis elements by $j_{N,m}$. These are considered to be generalizations of the $j_m$ functions, because $j_{1,m}(\t)=j_m(\t)+24\sigma_1(m)$ for any non-negative integer $m$ \cite[p.99]{JKK-sesqui}. 
The harmonic Maass function $j_{N,m}$ is decomposed into  (\cite[Theorem 1]{Niebur}, \cite[Proposition 3.1]{BKLOR})
\begin{equation}\label{jnnfour}
j_{N,m}(\t)=j^{h}_{N,m}(\t)+j^{nh}_{N,m}(\t),
\end{equation}
where the Fourier expansion of the mock modular function $j^h_{N,m}$ is given by 
\begin{equation}\label{mockj}j^h_{N,m}(\t) =q^{-m}+c_N(m,0)+\sum_{n=1}^\infty c_N(m,n)q^n
\end{equation} and its corresponding non-holomorphic part $j^{nh}_{N,m}$ can be written as
$$j^{nh}_{N,m}(\t)=-\overline{q}^{m}+\sum_{n=1}^\i c_N(m,-n)\overline{q}^n.$$
The functions $j_m$ form a Hecke system, that is, for the normalized Hecke operator $T(m)$, they satisfy that
\begin{equation}\label{hj}j_m(\t)=j_1|T(m)(\t).
\end{equation}
In \cite{JKK-hecke}, a Hecke system for $j_{N,m}$ is constructed with a suitable generalization of the Hecke operator in \eqref{hj}.  By abuse of notation, we may write for any positive integers $N$ and $m$,  \cite[Corollary 1.4]{JKK-hecke}
\begin{equation}\label{hjN}
j_{N,m}(\t)=j_{N,1}|T(m)(\t).
\end{equation}
Using this property and its applications, one can obtain several congruences for the Fourier coefficients of mock modular function $j^h_{N,m}$.

\begin{thm}\label{pnn} Let $p$ be a prime and $c_p(m,n)$ be the Fourier coefficients of $j^h_{p,m}(\t)$, the mock modular function of level $p$ given in \eqref{mockj}. Then  
\begin{enumerate}
\item $pc_{p}(m,pn)-c_p(pm,n)\in \mathbb Z,$
\item if $p\nmid n$ and $c_p(m,pn)\in \mathbb Z$, then 
$c_{p}(pm,n)\equiv 0 \pmod p.$
\end{enumerate}
\end{thm}

Since $j_{N,m}$ is weakly holomorphic when the level $N$ is of genus $0$,  for each $p\in\{2,3, 5, 7\}$, $j_{p,m}=j^h_{p,m}$, and also it equals $f_{p,m}$ in Theorem~\ref{ajp} up to a constant. Utilizing congruences in \cite{Griffin} into the identity induced from \eqref{hjN} gives congruences of $c_p(m,n)$ for the primes $p$, but they can be obtained from Theorem~\ref{ajp} as well. The following congruence for the Fourier coefficients of a mock modular form of level $11$ that we derive from Theorem \ref{pnn} using Theorem~\ref{tb11} is new.

\begin{cor}\label{tc11} 
If $m, n\not\equiv 0\pmod{11}$ and non-negative integers $\a$ and $\b$ satisfy $\b>\a+1$, then
\begin{equation}\label{ec11}
11c_{11}(11^{\a}m,11^{\b+1} n)- c_{11}(11^{\a+1}m,11^\b n) \equiv 0  \pmod{11^{\b-\a-1}}.
\end{equation}
\end{cor}
Furthermore, by the relation between $f_{N,m}(\t)$ and $j_{N,m}(\t)$ provided in \eqref{fj} below, one can find more congruences for $c_p(m,n)$ from Theorems~\ref{tf11} and \ref{Tcong}.

The rest of the paper is organized as follows. In Section 2, we introduce the {\it{$p$-plication formula}} for the harmonic Maass function $j_{N,m}$ proved in \cite{JKK-hecke} and congruence relations found in \cite{Griffin}. Then we prove Theorem \ref{pnn} and Corollary \ref{tc11}. We also prove Theorem \ref{tb11} following the method in \cite{Griffin}. In Section 3, we present the basis element $f_{N,m}$ and discuss shortly its congruence property in Theorem \ref{Tcong}. In Section 4, utilizing the $U$-operator and applying Atkin's results in \cite{Atkin1967}, we prove Theorem \ref{tf11}. 

\section{Proofs of Theorems \ref{tb11} and \ref{pnn} and Corollary \ref{tc11}}
For any positive integer $N$, let $M^!_k(N)$ (resp. $H_k(N)$) denote the space of weight $k$ weakly holomorphic modular forms (resp. harmonic Maass forms) on $\G_0(N)$.  
Also, let $M_k^{!,\infty}(N)$  (resp. $M_k^{!,0}(N)$) be its subspace of modular forms whose poles are supported only at infinity (resp. $0$) and holomorphic at other cusps.  For a prime $p$, the $U_p$ operator that acts on a complex valued function on $\H$ is defined by
\begin{equation}\label{ups}U_pf(\t)=\frac1p\sum_{i=0}^{p-1} f\left( \frac{\t+i}{p} \right)\end{equation} 
and $U_p^{n+1}f(\t)=U_p(U_p^nf(\t))$ for $n\geq 1$.
Also for a non-negative integer $i$, the {\it{$p$-plication of $j_{N,m}$}} is given by 
\begin{equation}\label{defp}j_{N,m}^{(p^i)}:=j_{\frac{N}{(N,p^i)},m},\end{equation}
where $j_{N,\frac{m}{p}}=0$ if $p\nmid m$. In \cite{JKK-hecke}, the following {\it{$p$-plication formula}} is established, which was essential in constructing the generalized Hecke operator for $j_{N,m}$ satisfying \eqref{hjN}.
\begin{prop}\cite[Theorem 1.1]{JKK-hecke} \label{p-pli} For any prime $p$, we have that
\begin{equation}\label{ep-pli}
p(U_pj_{N,m})(\t)+j_{N,m}^{(p)}(p\t)=j_{N,pm}(\t)+pj_{N,\frac{m}{p}}^{(p)}(\t).
\end{equation}
\end{prop}

\begin{proof}[Proof of Theorem \ref{pnn}]
Let $N=p$ in \eqref{ep-pli} so that after rearrangement of terms, we have by \eqref{defp} that

\begin{equation}\label{pp-pli}
p(U_pj_{p,m})(\t)-j_{p,pm}(\t)=p j_{1,\frac{m}{p}}(\t)-j_{1,m}(p\t).
\end{equation}
Note that the function on the right-hand side of \eqref{pp-pli} is weakly holomorphic, because $j_{1,m}=j_{m}+24\sigma_1(m)$. Hence we can decompose \eqref{pp-pli} into two identities:
\begin{equation}\label{ppp-pli}
p(U_pj^h_{p,m})(\t)-j^h_{p,pm}(\t)=p j_{1,\frac{m}{p}}(\t)-j_{1,m}(p\t)
\end{equation}
and
\begin{equation}\label{ppm-pli}
p(U_pj^{nh}_{p,m})(\t)=j^{nh}_{p,pm}(\t).
\end{equation}
By comparing the coefficients of $q^n$ in both sides of \eqref{ppp-pli}, we find that
\begin{equation}\label{ppcoef-pli}
pc_{p}(m,pn)-c_{p}(pm,n)=p c_1(\frac{m}{p},n)-c_1(m,\frac{n}{p}).
\end{equation}
As the coefficients of $j_{1,m}(\t)$ are integers, so are the Fourier coefficients of the mock modular function $p(U_pj^h_{p,m})(\t)-j^h_{p,pm}(\t)$ in the left-hand side of \eqref{ppcoef-pli}. Theorem \ref{pnn} immediately follows from \eqref{ppcoef-pli}.  
\end{proof}

Applying known congruences of  $c_1(m,n)=b(m,n)\ (n>0)$ into \eqref{ppcoef-pli}, one can obtain many congruences for $pc_{p}(m,pn)-c_{p}(pm,n)$.
In particular, we recall Griffin's congruences on the Fourier coefficients of $j_{m}(\t)$, the most general congruences for $j_m(\t)$ known. He used Zagier duality between the Fourier coefficients of weights $k$ and $2-k$ weakly holomorphic modular forms and properties of Hecke operator.  
\begin{thm}\cite[Theorem 2.1]{Griffin}\label{G}
For each $p\in\{2,3,5,7\}$, let $\a,\b\geq 0$, $\g=|\a-\b|$, and $m,n\not\equiv 0\pmod p$. Then\\
For $p=2\ \mathrm{:}$
\begin{align*}
b(2^{\a}m,2^{\b}n)&\equiv -2^{3\g+8}3^{\g-1}m\sigma_7(m)\sigma_7(n)  \pmod{2^{3\g+13}} &\ \text{if}\ \b>\a,\\ \nonumber
&\equiv -2^{4\g+8}3^{\g-1}m\sigma_7(m)\sigma_7(n)  \pmod{2^{4\g+13}}&\ \textit{if}\  \a>\b,\\ \nonumber
&\equiv 20m\sigma_7(m)\sigma_7(n) \pmod{2^7} &\ \textit{if}\  \g=0\ and\ mn\equiv 1\pmod 8,\\
&\equiv \frac12m\sigma_1(m)\sigma_1(n) \pmod{2^3} &\ \textit{if}\  \g=0\ and\ mn\equiv 3\pmod 8,\\ \nonumber
&\equiv -12m\sigma_7(m)\sigma_7(n) \pmod{2^8} &\ \textit{if}\  \g=0\ and\ mn\equiv 5\pmod 8.\nonumber
\end{align*}
For $p=3\ \mathrm{:}$
\begin{align*}
b(3^{\a}m,3^{\b}n)&\equiv \mp 3^{2\g+3}10^{\g-1}\frac{\sigma_1(m)\sigma_1(n)}{n}  \pmod{3^{2\g+6}} &\ \text{if}\ \b>\a\ and\ mn\equiv \pm1\pmod 3,\\ \nonumber
&\equiv \mp 3^{3\g+3}10^{\g-1}\frac{\sigma_1(m)\sigma_1(n)}{n} \pmod{3^{3\g+6}} &\ \text{if}\ \a>\b\ and\ mn\equiv \pm1\pmod 3,\\
&\equiv 2\cdot3^{3}\frac{\sigma_1(m)\sigma_1(n)}{n}  \pmod{3^{7}} &\ \text{if}\ \g=0\ and\ mn\equiv \pm1\pmod 3.\nonumber
\end{align*}
For $p=5\ \mathrm{:}$
\begin{align*}
b(5^{\a}m,5^{\b}n)&\equiv -5^{\g+1}3^{\g-1}m^2n\sigma_1(m)\sigma_1(n)  \pmod{5^{\g+2}} &\ \text{if}\ \b>\a,\\ \nonumber
&\equiv -5^{2\g+1}3^{\g-1}m^2 n\sigma_1(m)\sigma_1(n)  \pmod{5^{2\g+2}} &\ \text{if}\ \a>\b,\\
&\equiv 10m^2n\sigma_1(m)\sigma_1(n)  \pmod{5^{2}} &\ \text{if}\ \g=0\ and\ \left(\frac{mn}{5}\right)=-1.\nonumber
\end{align*}
For $p=7\ \mathrm{:}$
\begin{align*}
b(7^{\a}m,7^{\b}n)&\equiv 7^{\g}5^{\g-1}m^2n\sigma_3(m)\sigma_3(n)  \pmod{7^{\g+1}} &\ \text{if}\ \b>\a,\\ \nonumber
&\equiv 7^{2\g}5^{\g-1}m^2n\sigma_3(m)\sigma_3(n)  \pmod{7^{2\g+1}} &\ \text{if}\ \a>\b,\\
&\equiv 2m^2n\sigma_3(m)\sigma_3(n)  \pmod{7} &\ \text{if}\ \g=0\ and\ \left(\frac{mn}{7}\right)=1.\nonumber
\end{align*}
\end{thm}

Utilizing congruences in Theorem \ref{G} into  \eqref{ppcoef-pli} for $p\in\{2,3,5,7\}$, one can find congruences of $c_p(m,n)$ for corresponding $p$'s, which are immediate consequences of those in Theorem~\ref{ajp}. 

\begin{proof}[Proof of Corollary \ref{tc11}] Similarly, Corollary~\ref{tc11} follows from Theorem~\ref{tb11} and \eqref{ppcoef-pli}.
\end{proof}

\begin{proof}[Proof of Theorem \ref{tb11}]
Recall from standard formulas for the action of the Hecke operator that if $(k,m)=1$, then
$$j_{m}|T(k)(\t)=\sum_n\left(\sum_{d|(k,n)}\frac{k}{d}b(m,\frac{kn}{d^2})\right)q^n.$$
Hence we have $$b(mk,n)=\sum_{d|(k,n)}\frac{k}{d}b(m,\frac{kn}{d^2}).$$
Taking $m=1$ and $k=m'$ not divisible by $11$, we find by using \eqref{j11} that 
\begin{equation}\label{111}b(m',11^\b n)=\sum_{d|(m',n)}\frac{m'}{d}b(1,11^\b\frac{m'n}{d^2})=\sum_{d|(m',n)}\frac{m'}{d}b(11^\b\frac{m'n}{d^2})\equiv 0\pmod{11^\b}.
\end{equation}
On the other hand, if we apply $T(11^\a)$ to $j_{m'}(\t)$, then we obtain
\begin{equation}\label{112}b(11^\a m',11^\b n)=\sum_{d|11^\a}\frac{11^\a}{d}b(m',n\frac{11^{\a+\b}}{d^2}).
\end{equation}
The theorem now follows from \eqref{111} and \eqref{112}.
\end{proof}

\section{Basis elements $f_{N,m}$ and Theorem~\ref{Tcong}}

Let $N=11,14,15,17,19,20,21,24,27,32,36,49$. These are the values of $N$ precisely when the modular curves $X_0(N)$ are of genus 1. For each of these $N$, there are two canonical generators of $\mathbb C$-algebra $M_0^{!,\infty}(N)$ of the form $q^{-2}+O(q^{-1})$ and $q^{-3}+O(q^{-1})$, because the infinity point is not a Weierstrass point of the Riemann surface $X_0(N)$.  Using these, one can construct a basis for $M_0^{!,\infty}(N)$  which consists of unique modular functions with integral Fourier coefficients in the form of
$f_{N,m}=q^{-m}+a_N(m,-1)q^{-1}+\sum_{n=1}^\infty a_N(m,n)q^n\in M_0^{!,\infty}(N)$
with $f_{N,0}(\t)=1$, $f_{N,1}(\t)=0$ and $a_N(1,-1)=-1$.
For example, when $N=11$, the smallest $N$ for which $X_0(N)$ is of genus 1,  two generators of $\C(X_0(11))$ are found by Y.~Yang \cite{Yang}. They are 
\begin{equation*}\label{x}
X= q^{-2}+2q^{-1}+4+5q+8q^2+q^3+7q^4-11q^5+\cdots
\end{equation*} and 
\begin{equation*}\label{y}Y= q^{-3}+3q^{-2}+7q^{-1}+12+17q+26q^2+19q^3+37q^4-15q^5+\cdots.
\end{equation*}
We let $f_{11,2}:= X-4$ and $f_{11,3}:= Y-3X$. One can continue to construct modular functions of the form $f_{11,m}=q^{-m}+O(q^{-1})=Q_m(X,Y)$, where $Q_m(X,Y)$ is a polynomial in $X$ and $Y$ that eliminates  the terms between $q^{-m}$ and $q^{-1}$ and also the constant term if there is any.  The first few of $f_{11,m}(\t)$ are given by
\begin{eqnarray}\label{0basis}
f_{11,0}&=&1\cr
f_{11,1}&=&q^{-1}-q^{-1}+0q+0q^2+0q^3+0q^4+0q^5+\cdots=0\cr
f_{11,2}&=&q^{-2}+2q^{-1}+5q+8q^2+q^3+7q^4-11q^5+\cdots=X-4\cr
f_{11,3}&=&q^{-3}+q^{-1}+2q+2q^2+16q^3+16q^4+18q^5+\cdots=Y-3X\cr
f_{11,4}&=&q^{-4}-2q^{-1}+6q+3q^2+18q^3-42q^4+\cdots=X^2-4Y-4X-36\\
f_{11,5}&=&q^{-5}-q^{-1}-14q-16q^2+34q^3+\cdots=XY-2X^2+X+7Y+60\cr
%f_{11,6}&=&q^{-6}-2q^{-1}-q+26q^2-93q^3+91q^4+39q^5+\cdots\cr
& \vdots&\cr
\end{eqnarray}
Using properties of the generalized Hecke operator appearing in \eqref{hjN}, several congruences of the Fourier coefficients of $f_{N,m}(\t)$ are proved in \cite{JKK-hecke}. Theorem \ref{Tcong} is one of them, which gives a strong divisibility property of Fourier coefficients of modular functions that holds for arbitrary prime powers.

Let $J_{N,m}(\t):=j_{N,m}(\t)-c_N(m,0)$ and define  $J_{N,0}(\t):=1$.
As $H_0(N)$ is generated by $j_{N,m}$, each of $f_{N,m}$ can be written as a linear combination of $j_{N,m}$'s. Specifically,
\begin{equation}\label{fj}f_{N,m}=J_{N,m}+a_{N}(m,-1)J_{N,1}.\end{equation}
We also define $f_{N,\frac{m}{p}}=0$ unless $p$ divides $m$. For more congruences of the Fourier coefficients of modular functions of arbitrary level, see \cite{JKK-hecke}.

\section{Proof of Theorem \ref{tf11}}
The method of proof given in this section is very similar to those in \cite{AJ, Atkin1967,Lehner1949a, Lehner1949}. In \cite{Atkin1967}, Atkin constructed bases $G_m$'s for $M_0^{!,\infty}(11)$ and $g_m$'s for $M_0^{!,0}(11)$. 
\begin{lem}\cite[Lemma 3]{Atkin1967}\label{A11}
For all integers $m\geq 2$, $G_m(\t)$ and $g_m(\t)$ satisfy the following properties $\mathrm{:} $
\begin{enumerate}
\item $G_m(\t)\in M_0^{!,\infty}(11)$ and $g_m(\t)\in M_0^{!,0}(11),$

\item $G_m(-1/11\t)=11^{\theta(m)}g_m(\t),$
where $\theta(m)=6k+2,3,4,6,6$ according as $m=5k+2,3,4,5,6$ ($k\geq 0$),

\item the Fourier coefficients of $G_m(\t)$ have  integral coefficients with leading term $q^{-m}$,

\item the Fourier coefficients of $g_m(\t)$ have integral coefficients with leading term $q^{\psi(m)}$,
 where $\psi(m)=5k+1,2,3,5,4$ according as  $m=5k+2,3,4,5,6$ ($k\geq 0$).
 \end{enumerate}
 
 Further, the function $B(\t)=G_2(\t)g_2(\t)-12\in M_0^{!}(11)$ has simple poles at $\t=0$ and $\t=i\i$ such that $B(-1/11\t)=B(\t)$ and the Fourier coefficients of $B(\t)$ have integral coefficients with leading term $q^{-1}$.
\end{lem}
According to \cite[Lemma 4]{Atkin1967}, $G_m(\t)$, $11^{\theta(m)}g_m(\t)$ and $B(\t)$ form a basis for $M_0^!(11)$. Thus if $F(\t)\in M_0^!(11)$ has a pole of order $M$ at $\t=0$ and a pole of order $K$ at $\t=i\i$, then
\begin{equation}\label{basis}
F(\t)=\sum_{r=2}^K\lambda_{-r}G_r(\t)+\lambda_{-1}B(\t)+\lambda_0+\sum_{r=2}^M \lambda_r 11^{\theta(r)} g_r(\t)
\end{equation}
for some constants $\lambda_r$ $(-K\leq r\leq M)$.

The Fourier expansions of $G_m(\t)$, $g_m(\t)$ and $B(\t)$ are found in \cite[Table 1]{Atkin1967}. 
%\begin{eqnarray}\label{Gbasis}
%G_{2}(\t)&=&q^{-2}+2q^{-1}-12+5q+8q^2+q^3+7q^4-11q^5+\cdots\cr
%G_{3}(\t)&=&q^{-3}-3q^{-2}-5q^{-1}+24-13q-22q^2+13q^3-5q^4+\cdots\cr
%G_{4}(\t)&=&q^{-4}-7q^{-3}+13q^{-2}+17q^{-1}-84+57q+93q^2-81q^3-\cdots\cr
%G_{5}(\t)&=&q^{-5}-12q^{-4}+54q^{-3}-88q^{-2}-99q^{-1}+540-418q-648q^2+\cdots\\
%G_{6}(\t)&=&q^{-6}-5q^{-5}-13q^{-4}+132q^{-3}-233q^{-2}-305q^{-1}+1404-910q-\cdots\cr
%& \vdots&\cr
%\end{eqnarray}
%and
%\begin{eqnarray*}\label{2basis}
%g_{2}(\t)&=&q+5q^2+19q^3+63q^4+185q^5+\cdots\cr
%g_{3}(\t)&=&q^2+9q^3+49q^4+214q^5+800q^6+\cdots\cr
%g_{4}(\t)&=&q^3+14q^4+102q^5+561q^6+2563q^7+\cdots\cr
%g_{5}(\t)&=&q^{5}+12q^6+90q^7+520q^8+2535q^9+\cdots\\
%g_{6}(\t)&=&q^{4}+19q^5+191q^6+1400q^7+8373q^8+\cdots\cr
%& \vdots&\cr
%\end{eqnarray*}
We note that
$G_2(\t)=f_{11,2}(\t)-12=X-16$ and $G_3(\t)=f_{11,3}(\t)-3f_{11,2}(\t)+24=Y-6X+36$. 
Also, note that  $\displaystyle{g_5(\t)=\frac{\eta^{12}(11\t)}{\eta^{12}(\t)}}$ and 
$\displaystyle{G_5(\t)=\frac{\eta^{12}(\t)}{\eta^{12}(11\t)}},$
where $\eta(\t)$ is the Dedekind eta-function $\eta(\t)=q^{1/24}\prod_{n=1}^\i(1-q^n)$.
In general, if we write $G_m(\t)=q^{-m}+\sum_{\ell=-m+1}^\i a_m(\ell)q^\ell$, then 
\begin{equation}\label{Gf}G_m(\t)=f_{11,m}(\t)+a_m(0)+\sum_{\ell=2}^{m-1} a_m(-\ell)f_{11,\ell}(\t),\quad (m\geq 2).
\end{equation}
Conversely, we can also express the element $f_{11,m}(\t)$ of the reduced form basis of $M_0^{!,\infty}(11)$ as a linear combination of $G_\ell(\t)$'s with $2\leq\ell\leq m$. Thus the congruence \eqref{e11} is equivalent to the same congruence of coefficients $a_m(\ell)$ of $G_m(\t)$. More precisely, $a_{11^\a m'}(11^{\b}\ell)\equiv 0 \pmod{11^{\b-\a}}$.  

In \cite{Atkin1967}, Atkin showed the divisibility property of the Fourier coefficients of $g_m$'s. Using this and the fact $U_{11}j(\t)\in M_0^{!,0}$ with a pole of order $121$ at $\t=0$, he proved  \eqref{j11}.
From now till the end of the paper, we use the notation $U:=U_{11}$.
\begin{lem}\cite[Theorem 8 and (8.81)]{Lehner1943}\label{L0} Let $f(\t)\in M_0^!(11)$.  Then
\begin{enumerate}
\item $Uf(\t)\in M_0^!(11)$,
\item $\displaystyle{11(Uf)(-\frac{1}{11\t})=11(Uf)(11\t)+f(-\frac{1}{121\t})-f(\t).}$
\end{enumerate}
\end{lem}

\begin{lem}\cite[Corollary, p.~21]{Atkin1967}\label{A28} Define $\xi(2)=0, \xi(3)=1$ and $\xi(n)=5k+1,3,3,4,5$ according as $n=5k+4,5,6,7,8$ ($k\geq 0$), We denote by $\mathcal S$ the class of functions $F(\t)$ with 
$\displaystyle{F(\t)=\sum_{n=2}^M\lambda_n11^{\xi(n)}g_n(\t).}$ Then whenever $F(\t)\in \mathcal S$, we have $11^{-1}UF(\t)\in \mathcal S$. 
\end{lem}
We now let $m=11^\a m'$ when $(11,m')=1$ and want to prove the congruence  $a_{11}(m,11^{\a+\g} n)\equiv 0 \pmod{11^\g}$ for any positive integer $\g$ by using induction on $\a$. For brevity, we let $f_m:=f_{11,m}$.
Consider
\begin{equation}\label{Hm}H_m(\t):=11(Uf_m)(\t)-11 f_{m/11}(\t).\end{equation}
Here $f_{m/11}(\t)=0$ when $11\nmid m$. 
By Lemma \ref{L0} (2), we have
\begin{eqnarray}\label{hmi}
H_m(-\frac{1}{11\t})&=&11(Uf_m)(-\frac{1}{11\t})-11 f_{m/11}(-\frac{1}{11\t})\cr
&=&11(Uf_m)(11\t)+f_m(-\frac{1}{121\t})-f_m(\t)-11 f_{m/11}(-\frac{1}{11\t})\\
&=&11q^{-m}+O(1)+O(1)-q^{-m}-O(q^{-1})-O(1)\cr
&=&10q^{-m}+O(q^{-1}).\nonumber
\end{eqnarray}
Apparently, when $\a=0$, $H_m(\t)$ is holomorphic at $\t=i\i$ and has a pole of order $m$ at $\t=0$. Hence 
it follows from \eqref{basis} that
$$Uf_m(\t)=\lambda_0+\sum_{r=2}^m \lambda_r 11^{\theta(r)-1}g_r(\t)$$ 
for some constants $\lambda_0$ and $\lambda_r$ ($2\leq r\leq m$). 
Thus every coefficient of $q^{11n}$ in $f_m(\t)$ is a multiple of $11$. Moreover, by Lemma \ref{A28}, after applying $U$-operator $\g$ times, we find that
$$a_{11}(m,11^\g n)\equiv 0 \pmod{11^\g}$$ for each positive integer $\g$.

Nextly, if $\a=1$ and $m'=1$, i.e., $m=11$, $H_{11}(\t)$  has a simple pole at $\t=i\i$ and a pole of order $11$ at $\t=0$ so that we obtain from \eqref{basis} that
\begin{equation}\label{b11}Uf_{11}(\t)=\lambda_{-1}B(\t)+\lambda_0+\sum_{r=2}^{11} \lambda_r 11^{\theta(r)-1}g_r(\t)\end{equation}
for some constants $\lambda_r$ ($-1\leq r\leq 11$). As $UB(\t)+5\in \mathcal S$ by \cite[p.~22]{Atkin1967},  every coefficient of $q^{11n}$ in $B(\t)$ is a multiple of $11$. We hence deduce from \eqref{b11} that every  coefficient of $q^{121n}$ in $f_{11}(\t)$ is a multiple of $11$. Again using Lemma \ref{A28} repeatedly, we find that 
$$a_{11}(11,11^{\g+1} n)\equiv 0 \pmod{11^\g}.$$ 

Now assume the congruence holds for all $m$ of the form $m=11^a m'$ with $a<\a$. Then it remains to prove that it holds for $m=11^\a m'$ as well.
In this general case, $H_m(\t)$ has a pole of order $m/11$ at $\t=i\i$ and a pole of order $m$ at $\t=0$.
Therefore, by \eqref{basis}
$$Uf_m(\t)=\sum_{r=2}^{m/11} \lambda_{-r} G_r(\t)+\lambda_{-1}B(\t)+\lambda_0+f_{m/11}(\t)+\sum_{r=2}^m \lambda_r 11^{\theta(r)-1}g_r(\t)$$ 
for some constants $\lambda_r$ ($-m/11\leq r\leq m$). By induction hypothesis and the arguments so far, 
we see that every coefficient of $q^{11^{\a+\g-1}n}$ in $Uf_m(\t)$ is a multiple of $11^\g$, which implies the desired result
$$a_{11}(m,11^{\a+\g} n)\equiv 0 \pmod{11^\g}.$$
Replacing $\g$ by $\b-\a$, we have \eqref{e11}.

\section*{Acknowledgements}

The author is grateful to the referee for valuable comments that helped to improve the article.

\end{document}